\documentclass[12pt]{article}
\usepackage{booktabs}
\usepackage{caption}
\usepackage{mathrsfs}
\usepackage{amsmath}
\usepackage{amsfonts,amsthm,amssymb,mathrsfs,bbding}
\usepackage{float}
\usepackage{graphicx}
\usepackage{enumerate}
\usepackage{tikz}
\usepackage{caption}
\usepackage{rotating}
\allowdisplaybreaks[4]
\usepackage{tabularx}
\usepackage{cite}
\evensidemargin=\oddsidemargin\topmargin=-1.5cm

\newtheorem{thm}{Theorem}[section]
\newtheorem{prob}{Problem}

\newtheorem{lem}[thm]{Lemma}

\theoremstyle{definition}
\renewcommand\proofname{\bf Proof}
\addtocounter{section}{0}
\begin{document}

\title{\LARGE{\bf Spectral radius and (globally) rigidity of graphs in $R^2$} \footnote{This work is supported by the National Natural Science Foundation of China (Grant Nos. 11771141, 12011530064 and 11871391)}\setcounter{footnote}{-1}\footnote{\emph{Email addresses:} ddfan0526@163.com (D. Fan), huangxymath@163.com (X. Huang), huiqiulin@126.com (H. Lin)}}
\author{Dandan Fan$^{a,b}$, Xueyi Huang$^a$, Huiqiu Lin$^a$\thanks{Corresponding author.}\\[2mm]
\small\it $^a$ School of Mathematics, East China University of Science and Technology, \\
\small\it   Shanghai 200237, China\\[1mm]
\small\it $^b$ College of Mathematics and Physics, Xinjiang Agricultural University\\
\small\it Urumqi, Xinjiang 830052, China}
\date{}
\maketitle

{\flushleft\large\bf Abstract}
Over the past half century,  the rigidity of graphs in  $R^2$ has aroused a great deal of interest. Lov\'{a}sz and Yemini (1982) proved that every $6$-connected graph is rigid in $R^2$. Jackson and Jord\'{a}n (2005) provided a  similar vertex-connectivity condition for the globally rigidity of graphs in $R^2$. These results imply that a graph $G$ with algebraic connectivity  $\mu(G)>5$ is  (globally) rigid  in $R^2$. Cioab\u{a}, Dewar and Gu (2021) improved this bound, and proved that a graph $G$ with minimum degree $\delta\geq 6$  is  rigid in $R^2$ if $\mu(G)>2+\frac{1}{\delta-1}$, and is globally rigid in $R^2$ if $\mu(G)>2+\frac{2}{\delta-1}$. In this paper, we study the (globally) rigidity of graphs in $R^2$ from the viewpoint of adjacency eigenvalues.  Specifically, we provide sufficient conditions for a  2-connected  (resp. 3-connected) graph with given minimum degree to be rigid (resp. globally rigid) in terms of the spectral radius. Furthermore, we determine the unique graph attaining the maximum spectral radius among all minimally rigid graphs of order $n$.

\begin{flushleft}
\end{flushleft}

\section{Introduction}

Rigidity is the property of a structure that does not flex under an applied force.  It is well studied in discrete geometry and mechanics, and has various applications in  material science, engineering and biological science.  Roughly speaking, a rigid graph is an embedding of a graph in a Euclidean space which is structurally rigid. That is, a graph is rigid if the structure formed by replacing the edges by rigid rods and the vertices by flexible hinges is rigid. 

Formally, a \textit{$d$-dimensional framework} $(G,p)$ in $R^d$ is the combination of a finite graph $G=(V,E)$ and a map $p: V\rightarrow R^{d}$. Two $d$-dimensional frameworks $(G,p)$ and $(G,q)$ are \textit{equivalent} if $\|p(u)-p(v)\|=\|q(u)-q(v)\|$ holds for every $uv\in E$, where $\|\cdot\|$ denotes the Euclidean norm in $R^d$. Two $d$-dimensional frameworks $(G,p)$ and $(G,q)$ are \textit{congruent} if $\|p(u)-p(v)\|=\|q(u)-q(v)\|$ holds for every $u,v\in V$. A $d$-dimensional framework $(G,p)$ is \textit{generic} if the coordinates of its vertices are algebraically independent over the rationals. A graph $G$ is said to be \textit{rigid} in $R^d$ if for every generic $d$-dimensional framework $(G,p)$ there exists an $\varepsilon>0$ such that every $d$-dimensional framework $(G,q)$ equivalent to $(G,p)$ satisfying $\|p(u)-q(u)\|<\varepsilon$ for all $u\in V$ is actually congruent to $(G,p)$. A graph $G$ is called \textit{redundantly rigid} in $R^d$ if $G-e$ is rigid in $R^d$ for every $e\in E(G)$. A $d$-dimensional framework $(G, p)$ is \textit{globally rigid} if every framework that is equivalent to $(G, p)$ is congruent to $(G, p)$.  We say that a graph $G$  is \textit{globally rigid} in $R^d$  if  there exists a globally rigid generic $d$-dimensional framework $(G, p)$ (cf. \cite{Cioaba}).

In 1970, Laman \cite{Leman} provided a combinatorial characterization for rigid graphs in $R^2$. Since then, the rigidity of graphs has arouse a lot of interest. Especially, some sufficient conditions in terms of the vertex-connectivity or edge-connectivity for a graph to be rigid or globally rigid in $R^2$  have been successively discovered.  In 1982, Lov\'{a}sz and Yemini \cite{Yemini} constructed some $5$-connected non-rigid graphs, and proved that every $6$-connected graph is rigid. 
In 1992, Hendrickson \cite{Hendrickson} proved that every globally rigid graph with at least four vertices is $3$-connected and redundantly rigid. Later, Jackson and Jord\'{a}n \cite{Jackson} proved that every $6$-connected graph is globally rigid. Also, they  observed that a simple graph $G$ is (globally) rigid if $G$ is $6$-edge-connected, $G-v$ is $4$-edge-connected for every $v\in V(G)$, and $G-\{u,v\}$ is $2$-edge-connected for every $u,v\in V(G)$ \cite{Jackson-1}. Jackson, Servatius and Servatius \cite{JSS} applied the concept of essential connectivity (see \cite{Lai} for the definition) to the investigation of  (global) rigidity of graphs, and showed that every $4$-connected essentially $6$-connected graph is globally rigid. Recently, Gu, Meng, Rolek, Wang and Yu\cite{Gu-Meng} used discharging arguments to prove that every $3$-connected essentially $9$-connected graph is globally rigid. Naturally, we consider the following problem:
\begin{prob}\label{prob1}
Which spectral condition can guarantee the rigidity of a graph in $R^2$?
\end{prob}

Let $G$ be a graph. Denote by  $\delta(G)$ ($\delta$ for short) the minimum degree of  $G$, $D(G)$ the diagonal matrix of vertex degrees of $G$, and $A(G)$ the adjacency matrix of $G$. Then the matrix $L(G)=D(G)-A(G)$ is called the \textit{Laplacian matrix} of $G$. The second smallest eigenvalue of $L(G)$, denoted by $\mu(G)$, is known as the \textit{algebraic connectivity} of $G$. It is well known that the vertex-connectivity of $G$ is at least $\mu(G)$.
With regard to Problem \ref{prob1}, the results in Lov\'{a}sz and Yemini \cite{Yemini} and Jackson and Jord\'{a}n \cite{Jackson-1} imply that $G$ is (globally) rigid if $\mu(G)>5$. Recently, Cioab\u{a}, Dewar and Gu \cite{Cioaba} improved their bound and utilized the necessary conditions for packing rigid subgraphs to obtain spectral conditions for a graph to be (globally) rigid in terms of the algebraic connectivity. Concretely, they showed that a graph $G$ with $\delta(G)=\delta\geq 6$ is  rigid if $\mu(G)>2+\frac{1}{\delta-1}$, and is globally rigid if $\mu(G)>2+\frac{2}{\delta-1}$. It is interesting to extend the results on finding a condition for a graph $G$ to be (globally) rigid in terms of its spectral radius.

The largest eigenvalue of $A(G)$, denoted by $\rho(G)$, is called the \textit{spectral radius} of $G$. A graph is $k$-connected if removing fewer than $k$ vertices always leaves the remaining graph connected.
Let $B_{n,n_1}^{i}$ be the graph obtained from $K_{n_1}\cup K_{n-n_{1}}$ by adding $i$ independent edges between $K_{n_1}$ and $K_{n-n_{1}}$.  In this paper, we give some sufficient conditions for a graph to be (globally) rigid in $R^{2}$ in terms of the spectral radius.

\begin{thm}\label{thm::1.1}
Let $G$ be a 2-connected graph of the minimum degree $\delta\geq 6$ and order $n\geq 2\delta+4$. If $\rho(G)\geq \rho(B_{n,\delta+1}^{2})$, then $G$ is rigid unless $G\cong B_{n,\delta+1}^{2}$.
\end{thm}

%\begin{remark}
% A \emph{cover} of a graph $G$ is a collection $\mathcal{X}=\{X_1, X_2,\ldots, X_t\}$ of subsets of $V(G)$ such that $E(G)=E(X_1)\cup E(X_2)\cup \cdots\cup E(X_t)$. In 1982, Lov\'{a}sz and Yemini \cite{Yemini} obtained a characterization of
%rigid graphs. They proved that a graph $G$ is rigid of order $n$ if and only if $\sum_{X\in \mathcal{X}}(2|V(X)|-3)\geq 2n-3$ for all covers $\mathcal{X}$ of $G$. If $G$ contains a cut vertex $u$. Let $G_1,G_2,\ldots, G_s$ be the connected components of $G-u$ and $|V(G_i)|=n_i$, where $1\leq i\leq s$. Then $n=\sum_{i=1}^{s}n_{i}+1$. We choose a cover $\mathcal{X}=\{V(G_1+u), V(G_2+u),\ldots, V(G_s+u)\}$ of $G$. For $s\geq 2$, we have
%$\sum_{X\in \mathcal{X}}(2|V(X)|-3)=2(\sum_{i=1}^{s}n_i+1)-s-2\leq 2n-4$. This implies that $G$ is not rigid if $G$ contain cut vertices.
%\end{remark}
Hendrickson \cite{Hendrickson} proved that every globally rigid graph in $R^d$ with at least $d+2$ vertices is $(d+1)$-connected and redundantly rigid. Thus $3$-connected is a necessary condition for a graph to be globally rigid in $R^2$.

\begin{thm}\label{thm::1.2}
Let $G$ be a 3-connected graph of the minimum degree $\delta\geq 6$ and order $n\geq 2\delta+4$. If $\rho(G)\geq \rho(B_{n,\delta+1}^{3})$, then $G$ is globally rigid unless $G\cong B_{n,\delta+1}^{3}$.
\end{thm}

%\begin{remark}
%Hendrickson \cite{Hendrickson} proved that any globally rigid graph in $R^d$ with at least $d+2$ vertices is $(d+1)$-connected and redundantly rigid. Thus 3-connected is a trivial necessary condition for a graph to be globally rigid in $R^2$.
%\end{remark}

Note that a graph is rigid if and only if it has a minimally rigid spanning subgraph.
The graph $G$ is said to be \emph{minimally rigid} if $G$ is rigid, and $G-e$ is not rigid for every $e\in E(G)$. In 1970, Leman \cite{Leman} rediscovered the minimally rigid graphs in $R^2$ by using the edge count property. He found that a graph $G$ with $n$ vertices and $m$ edges is a minimally rigid  if and only if $m=2n-3$ and
$e_{G}(X)\leq 2|X|-3$ for all $X\subseteq V(G)$ with $|X|\geq 2$, where $e_{G}(X)$ is the number of edges of the subgraph $G[X]$ induced by $X$ in $G$. Minimally rigid graphs are also called Leman graphs in $R^2$. Based on the structural property of minimally rigid graphs, we determine the unique graph with the maximum spectral radius among all connected minimally rigid graphs of order $n$ in $R^2$ .

\begin{thm}\label{thm::1.3}
Let $G$ be a connected minimally rigid graph of order $n\geq 3$.
Then $\rho(G)\leq \rho(K_2\nabla (n-2)K_{1})$, with equality if and only if $G\cong K_2\nabla (n-2)K_{1}$.
\end{thm}

%\begin{thm}\label{thm::1.3}
%Let $G$ be a connected minimally rigid with order $n\geq 6$.
%Then $q(G)\leq q(K_2\nabla (n-2)K_{1})$, with equality if and only if $G\cong K_2\nabla (n-2)K_{1}$.
%\end{thm}

\section{Proof of Theorem \ref{thm::1.1}}

For any partition $\pi$ of $V(G)$, $e_{G}(\pi)$ denotes the number of edges of $G$ whose ends lie in different parts of $\pi$. A part is trivial if it contains of a single vertex. Let $Z\subset V(G)$ and $\pi$ be a partition of $V(G-Z)$ with $n_{0}$ trivial parts $v_1,v_2,\ldots, v_{n_0}$. Denote by $n_{Z}(\pi)=\sum_{1\leq i\leq n_0}|Z_i|$, where $Z_i$ is the set of vertices in $Z$ that are adjacent to $v_i$ for $1\leq i\leq n_0$. The following structural lemma will play an essential role in the proof of Theorem \ref{thm::1.1}.

\begin{lem}(See \cite{Gu-1})\label{lem::2.1}
A graph $G$ contains $k$ edge-disjoint spanning rigid subgraphs if for every $Z\subset V(G)$ and every partition $\pi$ of $V(G-Z)$ with $n_{0}$ trivial parts and $n_{0}'$ nontrivial parts,
$$e_{G-Z}(\pi)\geq k(3-|Z|)n_{0}'+2kn_0-3k-n_{Z}(\pi).$$
\end{lem}

For $X,Y\subset V(G)$, we denote by $E_{G}(X,Y)$ the set of edges with one endpoint in $U$ and one endpoint in $V$, and $e_{G}(X,Y)$ the number of edges in $E_{G}(X,Y)$. In particular, denote by $\partial_{G}(S)=E_{G}(S,V(G)-S)$.

\begin{lem}(See \cite{Gu-Lai})\label{lem::2.2}
Let $G$ be a graph with the minimum degree $\delta$ and $U$ be a non-empty proper subset of $V(G)$. If $|\partial_{G}(U)|\leq \delta-1$, then $|U|\geq \delta+1$.
\end{lem}

%\begin{lem}\label{lem::2.3}
%Suppose that $G$ is a graph with the minimum degree $\delta>3$ and $U\subset V(G)$ with $|U|\geq 2$. If $|\partial_{G}(U)|\leq \delta+1$, then $|U|\geq \delta$.
%\end{lem}
%\begin{proof}
%Suppose to the contrary that $|U|\leq \delta-1$. Then
%\begin{equation*}
%\begin{aligned}
%\delta|U|\leq|U|(|U|-1)+|\partial_{G}(U)|\leq (\delta-1)(|U|-1)+\delta+1=\delta|U|+2-|U|.
%\end{aligned}
%\end{equation*}
%If $|U|\geq 3$, then we can deduce a contradiction. We assume that $|U|=2$. Note that $|\partial_{G}(U)|\geq 2\delta-2$ and $\delta> 3$. Thus $|\partial_{G}(U)|> \delta+1$, a contradiction. This implies that $|U|\geq \delta$, as required.\end{proof}

Let $M$ be a real $n\times n$ matrix, and let $X=\{1,2,\ldots,n\}$. Given a partition $\Pi=\{X_1,X_2,\ldots,X_k\}$ with $X=X_{1}\cup X_{2}\cup \cdots \cup X_{k}$, the matrix $M$ can be partitioned as
$$
M=\left(\begin{array}{ccccccc}
M_{1,1}&M_{1,2}&\cdots &M_{1,k}\\
M_{2,1}&M_{2,2}&\cdots &M_{2,k}\\
\vdots& \vdots& \ddots& \vdots\\
M_{k,1}&M_{k,2}&\cdots &M_{k,k}\\
\end{array}\right).
$$
The \textit{quotient matrix} of $M$ with respect to $\Pi$ is defined as the $k\times k$ matrix $B_\Pi=(b_{i,j})_{i,j=1}^k$ where $b_{i,j}$ is the  average value of all row sums of $M_{i,j}$.
The partition $\Pi$ is called \textit{equitable} if each block $M_{i,j}$ of $M$ has constant row sum $b_{i,j}$.
Also, we say that the quotient matrix $B_\Pi$ is \textit{equitable} if $\Pi$ is an equitable partition of $M$.

\begin{lem}(Brouwer and Haemers \cite{BH}; Godsil and Royle\cite{C.Godsil})\label{lem::2.3}
Let $M$ be a real symmetric matrix, and let $\lambda_{1}(M)$ be the largest eigenvalue of $M$. If $B_\Pi$ is an equitable quotient matrix of $M$, then the eigenvalues of  $B_\Pi$ are also eigenvalues of $M$. Furthermore, if $M$ is nonnegative and irreducible, then $\lambda_{1}(M) = \lambda_{1}(B_\Pi).$
\end{lem}

Recall that $B_{n,n_1}^{i}$ is the graph obtained from $K_{n_1}\cup K_{n-n_{1}}$ by adding $i$ independent edges between $K_{n_1}$ and $K_{n-n_{1}}$.

\begin{lem}\label{lem::2.4}
Let $i\geq 2$, $a\geq i+1$ and $n\geq 2a+2$. Then
$$\rho(B_{n,a+1}^{i})<\rho(B_{n,a}^{i}).$$
\end{lem}
\renewcommand\proofname{\bf Proof}
\begin{proof}
Since $B_{n,a}^{i}$ contains $K_{n-a}$ as a proper subgraph, it follows that  $\rho(B_{n,a}^{i})>\rho(K_{n-a})=n-a-1$.
Observe that $A(B_{n,a}^{i})$ has the equitable quotient matrix
$$
C_{\Pi}^a=\begin{bmatrix}
i-1 &a-i & 1&0\\
i &a-(i+1) & 0&0\\
1 &0 & i-1&n-(a+i)\\
0 &0 & i&n-(a+i+1)\\
\end{bmatrix}.
$$
By a simple calculation, the characteristic polynomial of $C_{\Pi}^a$ is
\begin{equation*}
\begin{aligned}
\varphi(C_{\Pi}^a,x)=x^4\!+\!(4\!-\!n)x^3\!+\!(an\!-\!a^2\!-\!3n\!+\!5)x^2\!+\!2(an\!-\!a^2
\!-\!i\!-\!n\!+\!1)x\!-\!i^2\!+\!in\!-\!2i.
\end{aligned}
\end{equation*}
Also, note that $A(B_{n,a+1}^{i})$ has the equitable quotient matrix $C_{\Pi}^{a+1}$, which is obtained by replacing $a$ with $a+1$ in $C_{\Pi}^a$. Note that $n\geq 2a+2$. Then
\begin{eqnarray*}
\varphi(C_{\Pi}^{a+1},x)-\varphi(C_{\Pi}^a,x)=x(x+2)^2(n-(2a+1))>0
\end{eqnarray*}
for $x\geq n-a-1$, which implies that
$\lambda_{1}(C_{\Pi}^{a+1})<\lambda_{1}(C_{\Pi}^a)$. Combining this with Lemma \ref{lem::2.3}, we have
$$\rho(B_{n,a+1}^{i})<\rho(B_{n,a}^{i}).$$

This completes the proof. \end{proof}

\begin{lem}(See \cite{HSF,V.N})\label{lem::2.5}
Let $G$ be a graph on $n$ vertices and $m$ edges with $\delta\geq 1$. Then
$$\rho(G) \leq \frac{\delta-1}{2}+\sqrt{2 m-n \delta+\frac{(\delta+1)^{2}}{4}},$$
with equality if and only if $G$ is either a $\delta$-regular graph or a bidegreed graph
in which each vertex is of degree either $\delta$ or $n-1$.
\end{lem}

\begin{lem}(See \cite{HSF,V.N})\label{lem::2.6}
For nonnegative integers $p$ and $q$ with $2q \leq p(p-1)$ and $0 \leq x \leq p-1$, the function $f(x)=(x-1) / 2+\sqrt{2 q-p x+(1+x)^{2} / 4}$ is decreasing with respect to $x$.
\end{lem}

\begin{lem}\label{lem::2.7}
For $t=3,4$, let $n_1,\ldots,n_t$ be positive integers such that $n=\sum_{i=1}^{t}n_i$ and $n_t\geq\max\{n_1,n_2,\ldots,n_{t-1}\}$. If $n_i\geq a_{i}$ for $1\leq i\leq t-1$, then
$${n_1\choose 2}+{n_2\choose 2}+\ldots+{n_t\choose 2}\leq {a_1\choose 2}+{a_2\choose 2}+\ldots+{n-\sum_{1\leq i\leq t-1}a_{i}\choose 2},$$
with equality holding if and only if $(n_1,n_2,\ldots,n_t)=(a_1,a_2,\ldots,n-\sum_{1\leq i\leq t-1}a_{i})$.
\end{lem}

\begin{proof}
For $1\leq i\leq t-1$, let $n_i=a_{i}+x_{i}$, where $x_i\geq 0$. If $t=4$,
then $n=\sum_{i=1}^{4}n_{i}$ and $n_4\geq\max\{n_1,n_2,n_3\}\geq \max\{a_1,a_2,a_3\}$. Note that $n_4=n-\sum_{i=1}^{3}(a_{i}+x_i)$ and
$n\geq \sum_{i=1}^{3}(a_{i}+x_i)+\max\{a_1,a_2,a_3\}$. Then
\begin{equation*}
\begin{aligned}
&{a_1\choose 2}+{a_2\choose 2}+{a_3\choose 2}+{n-(a_1+a_2+a_3)\choose 2}\!-\!{n_1\choose 2}\!-\!{n_2\choose 2}\!-\!{n_3\choose 2}\!-\!{n_4\choose 2}\\
&=\sum_{i=1}^{3}x_i(n\!-\!\sum_{i=1}^{3}a_{i})\!-\!x_1a_1\!-\!x_2a_2\!-\!x_3a_3\!-\!((x_1\!+\!x_2\!+\!x_3)x_{1}\!+\!x_{2}^2\!+\!x_{3}x_{2}\!+\!x_{3}^2)\\
&\geq \sum_{i=1}^{3}x_i(n\!-\!\sum_{i=1}^{3}a_{i}\!-\!\max\{a_1,a_2,a_3\})\!-\!(x_1\!+\!x_2\!+\!x_3)^2~(\mbox{since $x_i\geq 0$ for $1\leq i\leq 3$})\\
&= \sum_{i=1}^{3}x_i(n-\sum_{i=1}^{3}(a_{i}+x_i)-\max\{a_1,a_2,a_3\})\\
&\geq 0,
\end{aligned}
\end{equation*}
with equality holding if and only if $x_1=x_2=x_3=0$.
If $t=3$, by using a similar analysis as above, we also get the result.\end{proof}

For any partition $\pi$ of $V(G)$, $e_{G}(\pi)$ denotes the number of edges of $G$ whose ends lie in different parts of $\pi$. For $X\subset V(G)$, let $G[X]$ be the subgraph of $G$ induced by $X$, and let $e_{G}(X)$ be the number of edges in $G[X]$. Now we shall give the proof of Theorem \ref{thm::1.1}.

\renewcommand\proofname{\bf Proof of Theorem \ref{thm::1.1}}
\begin{proof}
Suppose that $G$ is not rigid, by Lemma \ref{lem::2.1}, there exist some $Z\subset V(G)$ and some partition $\pi$ of $V(G-Z)$ with $n_{0}$ trivial parts $v_1,v_2,\ldots, v_{n_0}$ and $n_{0}'$ nontrivial parts $V_1,V_2,\ldots,V_{n_0'}$ such that
\begin{equation}\label{equ::1}
\begin{aligned}
e_{G-Z}(\pi)\leq (3-|Z|)n_{0}'+2n_0-4-n_{Z}(\pi),
\end{aligned}
\end{equation}
where $n_{Z}(\pi)=\sum_{1\leq i\leq n_0}|Z_i|$ and $Z_i$ is the set of vertices in $Z$ that are adjacent to $v_i$ for $1\leq i\leq n_0$. Note that $d_{G-Z}(v_i)\geq \delta-|Z_i|$. Then
\begin{equation}\label{equ::2}
\begin{aligned}
e_{G-Z}(\pi)&= \frac{1}{2}(\sum_{1\leq i\leq n_{0}'}|\partial_{G-Z}(V_i)|+\sum_{1\leq j\leq n_0}d_{G-Z}(v_j))\\
&\geq \frac{1}{2}(\sum_{1\leq i\leq n_{0}'}|\partial_{G-Z}(V_i)|+\delta n_0-\sum_{1\leq j\leq n_0}|Z_j|)\\
&\geq \frac{1}{2}(\sum_{1\leq i\leq n_{0}'}|\partial_{G-Z}(V_i)|+6n_0-n_{Z}(\pi))~(\mbox{since $\delta\geq 6$}),
\end{aligned}
\end{equation}
and so
\begin{equation}\label{equ::3}
\begin{aligned}
e_{G-Z}(\pi)&\geq 3n_{0}-\frac{1}{2}n_{Z}(\pi).
           \end{aligned}
\end{equation}
Thus, we have the following two claims.

{\flushleft\bf Claim 1.} $|Z|\leq 2$.

Otherwise, $|Z|\geq 3$. Thus
\begin{equation*}
\begin{aligned}
e_{G-Z}(\pi)\leq (3-|Z|)n_{0}'+2n_0-4-n_{Z}(\pi)\leq 2n_0-4-n_{Z}(\pi)
\end{aligned}
\end{equation*}
by (\ref{equ::1}). Combining this with (\ref{equ::3}), we have $n_0+4+\frac{1}{2}n_{Z}(\pi)\leq 0$, which is impossible because $n_{0}\geq 0$ and $n_{Z}(\pi)\geq 0$. This implies that $|Z|\leq 2$, as required.

{\flushleft\bf Claim 2.} $n_{0}'\geq 2$.

Otherwise, $n_{0}'\leq 1$. By Claim 1, we have $0\leq |Z|\leq 2$, and so
\begin{equation*}
\begin{aligned}
e_{G-Z}(\pi)\leq (3-|Z|)n_{0}'+2n_0-4-n_{Z}(\pi)\leq 2n_0-1-n_{Z}(\pi)
\end{aligned}
\end{equation*}
by (\ref{equ::1}). Combining this with (\ref{equ::3}), we have $n_0+1+\frac{1}{2}n_{Z}(\pi)\leq 0$, which is impossible because $n_{0}\geq 0$ and $n_{Z}(\pi)\geq 0$. This implies that $n_{0}'\geq 2$, as required.

Note that $\rho(G)\geq\rho(B_{n,\delta+1}^{2})>\rho(K_{n-\delta-1})=n-\delta-2$. Combining this with Lemmas \ref{lem::2.5} and \ref{lem::2.6}, we have
\begin{equation}\label{equ::4}
m > \frac{n^2}{2}-\frac{(2\delta+3)n}{2}+(\delta+1)^{2}.
\end{equation}
Since $G$ is a 2-connected graph, it follows that
\begin{equation}\label{equ::5}
|\partial_{G-Z}(V_i)|\geq 2-|Z|
\end{equation}
for $1\leq i\leq n_{0}'$. Recall that $0\leq |Z|\leq 2$. Then we divide the proof into the following two cases.

{\flushleft\bf Case 1.} $0\leq|Z|\leq 1$.

If $n_{0}'=2$, then the partition $\pi$ consists of two nontrivial parts $V_1,V_2$ and $n_0$ trivial parts. Putting (\ref{equ::5}) into (\ref{equ::2}), we get
$$e_{G-Z}(\pi)\geq \frac{1}{2}(|\partial_{G-Z}(V_1)|+|\partial_{G-Z}(V_2)|+6n_0-n_{Z}(\pi))\geq 2-|Z|+3n_0-\frac{1}{2}n_{Z}(\pi).$$
Combining this with (\ref{equ::1}) and $n_{0}'=2$, we have
$$-n_{0}-\frac{1}{2}n_{Z}(\pi)-|Z|\geq 0,$$
and so $n_{0}=0$, $n_{Z}(\pi)=0$ and $|Z|=0$ by the fact $n_{0}\geq0$, $n_{Z}(\pi)\geq0$ and $|Z|\geq 0$. This implies that the partition $\pi$ consists of two nontrivial parts $V_1,V_2$ and $G-Z=G$. Then $V(G)=V_1\cup V_2$ and $e_{G}(\pi)=e_{G}(V_1,V_2)\leq 2$ by (\ref{equ::1}). Notice that $e_{G}(V_1,V_2)=\frac{1}{2}(|\partial_{G}(V_1)|+|\partial_{G}(V_2)|)\geq 2$ by (\ref{equ::5}). Thus, $e_{G}(V_1,V_2)= 2$. Let $E_{G}(V_1,V_2)=\{f_1,f_2\}$. Then we assert that $f_1$ and $f_2$ are not incident. Otherwise, $f_1\cap f_2=\{u\}$. It is easy to see that $u$ is a cut vertex of $G$, which is impossible because $G$ is a 2-connected graph. One can verify that $G$ is a spanning subgraph of $B_{n,|V_1|}^{2}$. Then
$$\rho(G)\leq \rho(B_{n,|V_1|}^{2}),$$
with equality holding if and only if $G\cong B_{n,|V_1|}^{2}$.
Note that $|\partial_{G}(V_1)|=|\partial_{G}(V_2)|= 2<\delta-1$. Combining this with Lemma \ref{lem::2.2}, we obtain that $\min\{|V_1|,|V_2|\}\geq\delta+1$. By Lemma \ref{lem::2.4}, we have
$$\rho(G)\leq\rho(B_{n,\delta+1}^{2}),$$
with equality holding if and only if $G\cong B_{n,\delta+1}^{2}$.
It is impossible because
$\rho(G)\geq\rho(B_{n,\delta+1}^{2})$ and $G\ncong B_{n,\delta+1}^{2}$.

Let $\delta(G-Z)=\delta'$. Then $\delta'\geq\delta-|Z|$. If $n_{0}'\geq3$, then we assert that there exist at least two nontrivial parts $V_1,V_2$ of partition $\pi$ such
that $|\partial_{G-Z}(V_i)|\leq \delta'-1$ for $i=1,2$. Otherwise, there exists at most one nontrivial part $V_j$ of partition $\pi$ such that $|\partial_{G}(V_j)|\leq \delta'-1$ for some $1\leq j\leq n_0'$. Thus, $|\partial_{G}(V_i)|\geq \delta'$ for $1\leq i\neq j\leq n_0'$, and so
\begin{equation*}
\begin{aligned}
&2e_{G-Z}(\pi)\\
&= \sum_{1\leq i\leq n_0'}|\partial_{G-Z}(V_i)|+\sum_{1\leq j\leq n_0}d_{G-Z}(v_j)\\
&\geq (n_0'-1)\delta'+2-|Z|+\delta n_0-n_{Z}(\pi)~(\mbox{since $|\partial_{G-Z}(V_j)|\geq 2-|Z|$})\\
&\geq (n_0'-1)(\delta-|Z|)+2-|Z|+\delta n_0-n_{Z}(\pi)~(\mbox{since $\delta'\geq\delta-|Z|$})\\
&=2(3\!-\!|Z|)n_0'\!+\!4n_0\!-\!8\!-\!2n_{Z}(\pi)\!+\!(\delta\!-\!6\!+\!|Z|)n_{0}'\!+\!(\delta\!-\!4)n_0\!-\!\delta\!+\!10\!+\!n_{Z}(\pi)~\\
&\geq 2(3\!-\!|Z|)n_0'\!+\!4n_0\!-\!8\!-\!2n_{Z}(\pi)\!+\!2\delta\!-\!8\!+\!3|Z|\!+\!n_{Z}(\pi)~(\mbox{since $n_{0}'\geq 3$ and $n_0\geq 0$})\\
&> 2(3\!-\!|Z|)n_0'\!+\!4n_0\!-\!8\!-\!2n_{Z}(\pi) ~(\mbox{since $\delta\geq 6$, $n_{Z}(\pi)\geq 0$ and $0\leq|Z|\leq1$}),
\end{aligned}
\end{equation*}
which contradicts with (\ref{equ::1}). This implies that there exist two nontrivial parts $V_1,V_2$ of partition $\pi$ such
that $|\partial_{G-Z}(V_i)|\leq \delta'-1$ for $i=1,2$. Therefore, $|V_i|\geq \delta'+1$ for $i=1,2$ by Lemma \ref{lem::2.2}.
If $|Z|=0$, then $\delta'=\delta$, and so $|V_i|\geq \delta+1$ for $i=1,2$. Since $n_0'\geq 3$, we have $|V_3|\geq 2$ and $n_0\leq n-\sum_{i=1}^{3}|V_i|\leq n-2\delta-4$, and so
$$e_{G}(\pi)\leq 3n_{0}'+2n_0-4\leq 3\cdot 3+2(n-2\delta-4)-4=2n-4\delta-3$$
by (\ref{equ::1}). Combining this with Lemma \ref{lem::2.7}, we have
\begin{equation*}
\begin{aligned}
m&\leq \max\left\{{\delta+1\choose 2} +{n-\delta-3\choose 2}+{2\choose 2},2{\delta+1\choose 2} +{n-2\delta-2\choose 2}\right\}+e_{G}(\pi)\\
&\leq{\delta+1\choose 2} +{n-\delta-3\choose 2}+{2\choose 2}+e_{G}(\pi)\\
&\leq \frac{n^2}{2}-\frac{(2\delta+3)n}{2}+\delta^2+4.
\end{aligned}
\end{equation*}
Combining this with (\ref{equ::4}), we have $\delta<\frac{3}{2}$, which is impossible because $\delta\geq 6$.
We assume that $|Z|=1$ in the following. Note that $\delta'\geq\delta-1$. Then $|V_i|\geq\delta$. Since $n_{0}'\geq 3$, we have $|V_3|\geq 2$, $n_0\leq n-|Z|-\sum_{i=1}^{3}V_{i}\leq n-2\delta-3$ and $|\partial_{G}(Z)|-n_{Z}(\pi)\leq n-n_{0}-1$. Combining this with (\ref{equ::1}), we have
\begin{equation*}
\begin{aligned}
e_{G-Z}(\pi)+|\partial_{G}(Z)|&\leq 2n_{0}'+2n_0-4-n_{Z}(\pi)+|\partial_{G}(Z)|\\
&\leq 2n_{0}'+n_0+n-5\\
&\leq 2\cdot 3+(n-2\delta-3)+n-5\\
&=2n-2\delta-2.
\end{aligned}
\end{equation*}
By Lemma \ref{lem::2.7}, we get
\begin{equation*}
\begin{aligned}
m&\leq\max\left\{{\delta\choose 2}\!+\!{n\!-\!|Z|\!-\!\delta\!-\!2\choose 2}\!+\!{2\choose 2},2{\delta\choose 2}\!+\!{n\!-\!|Z|\!-\!2\delta\choose 2}\right\}\!+\!e_{G\!-\!Z}(\pi)\!+\!|\partial_{G}(Z)|\\
&\leq {\delta\choose 2}\!+\!{n\!-\!|Z|\!-\!\delta\!-\!2\choose 2}\!+\!{2\choose 2}+e_{G-Z}(\pi)+|\partial_{G}(Z)|\\
&\leq \frac{n^2}{2}-\frac{(2\delta+3)n}{2}+\delta^2+\delta+5.
\end{aligned}
\end{equation*}
Combining this with (\ref{equ::4}), we have $\delta<4$, which is impossible because $\delta\geq 6$.

{\flushleft\bf Case 2.} $|Z|=2$.

By (\ref{equ::1}), we have
\begin{equation}\label{equ::6}
\begin{aligned}
e_{G-Z}(\pi)\leq n_{0}'+2n_0-4-n_{Z}(\pi).
\end{aligned}
\end{equation}
If $2\leq n_{0}'\leq 3$, by (\ref{equ::2}), (\ref{equ::5}) and (\ref{equ::6}), we have
$$0\leq\sum_{1\leq i\leq n_{0}'}|\partial_{G-Z}(V_i)|\leq 2n_{0}'-8-2n_{0}-n_{Z}(\pi)\leq -2$$
by the fact $n_{Z}(\pi)\geq 0$ and $n_0\geq 0$, a contradiction.  Let $\delta(G-Z)=\delta'$. Then $\delta'\geq\delta-2$. If $n_{0}'\geq 4$, then we assert that there exist at least two nontrivial parts $V_1,V_2$ of partition $\pi$ such
that $|\partial_{G-Z}(V_i)|\leq \delta'-1$ for $i=1,2$. Otherwise, there exists at most one nontrivial part $V_j$ of partition $\pi$ such that $|\partial_{G-Z}(V_j)|\leq \delta'-1$ for some $1\leq j\leq n_0'$. Note that $|\partial_{G-Z}(V_i)|\geq \delta'$ for $1\leq i\neq j\leq n_{0}'$ and  $|\partial_{G-Z}(V_j)|\geq 0$. Then
\begin{equation*}
\begin{aligned}
2e_{G-Z}(\pi)&= \sum_{1\leq i\leq n_{0}'}|\partial_{G-Z}(V_i)|+\sum_{1\leq j\leq n_0}d_{G-Z}(v_j)\\
&\geq (n_0'-1)\delta'+\delta n_0-n_{Z}(\pi)\\
&\geq (n_0'-1)(\delta-2)+\delta n_0-n_{Z}(\pi)~(\mbox{since $\delta'\geq \delta-2$})\\
&=(2n_{0}'+4n_0-8-2n_{Z}(\pi))+(\delta-4)n_{0}'-\delta+(\delta-4)n_0+n_{Z}(\pi)+10\\
&\geq 2n_{0}'+4n_0-8-2n_{Z}(\pi)+3\delta-6~(\mbox{since $n_{0}'\geq 4$, $n_0\geq 0$ and  $n_{Z}(\pi)\geq 0$})\\
&>2n_{0}'+4n_0-8-2n_{Z}(\pi)~(\mbox{since $\delta\geq 6$}),
\end{aligned}
\end{equation*}
which contradicts with (\ref{equ::6}). This implies that there exist two nontrivial parts $V_1,V_2$ of partition $\pi$ such
that $|\partial_{G-Z}(V_i)|\leq \delta'-1$ for $i=1,2$. Therefore, by Lemmas \ref{lem::2.2}, we have $|V_i|\geq \delta'+1\geq\delta-1$ for $i=1,2$. Since $n_{0}'\geq 4$, $|V_3|\geq 2$ and $|V_4|\geq 2$, we have $n_{0}'\leq \frac{n-|Z|-2(\delta-1)-4}{2}+4=\frac{n}{2}-\delta+2$. Note that $|\partial_{G}(Z)|+e_{G}(Z)-n_{Z}(\pi)\leq 2(n-2-n_{0})+1$. Combining this with (\ref{equ::6}), we have
\begin{equation*}
\begin{aligned}
e_{G-Z}(\pi)\!+\!|\partial_{G}(Z)|\!+\!e_{G}(Z)&\leq n_{0}'+2n-7\leq\frac{5n}{2}-\delta-5.
\end{aligned}
\end{equation*}
By Lemma \ref{lem::2.7}, we get
\begin{equation*}
\begin{aligned}
m&\leq\max\left\{{\delta\!-\!1\choose 2}\!+\!2{2\choose 2}\!+\!{n\!-\!|Z|\!-\!\delta\!-\!3\choose 2},2{\delta\!-\!1\choose 2}\!+\!{2\choose 2}\!+\!{n\!-\!|Z|\!-\!2\delta\choose 2}\right\}\\
&~~\!+\!e_{G-Z}(\pi)\!+\!|\partial_{G}(Z)|\!+\!e_{G}(Z)\\
&\leq {\delta\!-\!1\choose 2}\!+\!2{2\choose 2}\!+\!{n\!-\!|Z|\!-\!\delta\!-\!3\choose 2}+e_{G-Z}(\pi)\!+\!|\partial_{G}(Z)|\!+\!e_{G}(Z)\\
&\leq \frac{n^2}{2}-\frac{(2\delta+6)n}{2}+\delta^2+3\delta+13.
\end{aligned}
\end{equation*}
Combining this with (\ref{equ::4}), we have $n<\frac{2}{3}\delta+8$, which is impossible because $n\geq 2\delta+4$ and $\delta\geq6$.

This completes the proof.\end{proof}

\section{Proof of Theorem \ref{thm::1.2}}

The following sufficient and necessary conditions for globally rigid in $R^2$ was established by Connelly \cite{Connelly}, Jackson and Jord\'{a}n \cite{Jackson-1}.

\begin{lem}(See \cite{Connelly,Jackson-1})\label{lem::3.1}
Let $G$ be a graph. Then $G$ is globally rigid if and only if either $G$ is a
complete graph on at most three vertices or $G$ is 3-connected and redundantly rigid.
\end{lem}

For any partition $\pi$ of $V(G)$, $E_{G}(\pi)$ denotes the edge set of $G$ whose ends lie in different parts of $\pi$. Now we shall give the proof of Theorem \ref{thm::1.2}.

\renewcommand\proofname{\bf Proof of Theorem \ref{thm::1.2}}
\begin{proof}
Suppose that $G$ is not  globally rigid. Note that $G$ is a 3-connected graph and $n\geq 2\delta+4$. Then $G$ is not redundantly rigid by Lemma \ref{lem::3.1}, there exists some $f\in E(G)$ such that $G-f$ is not rigid. According to Lemma \ref{lem::2.1}, there exist some $Z\subset V(G-f)$ and some partition $\pi$ of $V(G-f-Z)$ with $n_{0}$ trivial parts $v_1,v_2,\ldots, v_{n_0}$ and $n_{0}'$ nontrivial parts $V_1,V_2,\ldots,V_{n_0'}$ such that
\begin{equation}\label{equ::7}
\begin{aligned}
e_{G-f-Z}(\pi)\leq (3-|Z|)n_{0}'+2n_0-4-n_{Z}(\pi).
\end{aligned}
\end{equation}
If $f\notin E_{G}(\pi)$, then
$$e_{G-Z}(\pi)=e_{G-f-Z}(\pi)\leq (3-|Z|)n_{0}'+2n_0-4-n_{Z}(\pi).$$
By using the same analysis as Theorem \ref{thm::1.1}, we can deduce a contradiction. Thus, we consider the case of $f\in E_{G}(\pi)$ in the following. Note that $f\in E_{G}(\pi)$. Then $e_{G-f-Z}=e_{G-Z}(\pi)-1$, and so
\begin{equation}\label{equ::8}
\begin{aligned}
e_{G-Z}(\pi)\leq (3-|Z|)n_{0}'+2n_0-3-n_{Z}(\pi)
\end{aligned}
\end{equation}
by (\ref{equ::7}). Recall that $n_{Z}(\pi)=\sum_{1\leq i\leq n_0}|Z_i|$, where $Z_i$ is the set of vertices in $Z$ that are adjacent to $v_i$ for $1\leq i\leq n_0$. Note that $d_{G-Z}(v_i)\geq \delta-|Z_i|$. Then
\begin{equation}\label{equ::9}
\begin{aligned}
e_{G-Z}(\pi)&= \frac{1}{2}( \sum_{1\leq i\leq n_{0}'}|\partial_{G-Z}(V_i)|+\sum_{1\leq j\leq n_0}d_{G-Z}(v_j))\\
&\geq \frac{1}{2}(\sum_{1\leq i\leq n_{0}'}|\partial_{G-Z}(V_i)|+6n_0-n_{Z}(\pi))~(\mbox{since $\delta\geq 6$}),
\end{aligned}
\end{equation}
and so
\begin{equation}\label{equ::10}
\begin{aligned}
e_{G-Z}(\pi)\geq 3n_{0}-\frac{1}{2}n_{Z}(\pi)~.
\end{aligned}
\end{equation}
Thus, we have the following two claims.

{\flushleft\bf Claim 1.} $|Z|\leq 2$.

Otherwise, $|Z|\geq 3$, by (\ref{equ::8}), we get
\begin{equation*}
\begin{aligned}
e_{G-Z}(\pi)\leq (3-|Z|)n_{0}'+2n_0-3-n_{Z}(\pi)\leq 2n_0-3-n_{Z}(\pi).
\end{aligned}
\end{equation*}
Combining this with (\ref{equ::10}), we have  $n_0+\frac{1}{2}n_{Z}(\pi)+3\leq 0$, which is impossible because $n_{0}\geq 0$ and $n_{Z}(\pi)\geq 0$. This implies that $|Z|\leq 2$, as required.

{\flushleft\bf Claim 2.} $n_{0}'\geq 2$.

Otherwise, $n_{0}'\leq 1$. Note that $0\leq |Z|\leq 2$ and $n_{Z}(\pi)\geq 0$. Then
\begin{equation}\label{equ::11}
\begin{aligned}
e_{G-Z}(\pi)\leq (3-|Z|)n_{0}'+2n_0-3-n_{Z}(\pi)\leq 2n_0-n_{Z}(\pi)
\end{aligned}
\end{equation}
by (\ref{equ::8}). Combining (\ref{equ::10}) with (\ref{equ::11}), we have
\begin{equation*}
\begin{aligned}
n_{0}+\frac{1}{2}n_{Z}(\pi)\leq 0.
\end{aligned}
\end{equation*}
If $n_0\geq 1$ or $n_{Z}(\pi)\geq 1$, then we can deduce a contradiction. This implies that all the equalities hold in (\ref{equ::10}) and (\ref{equ::11}), and so $n_0=0$, $n_{0}'=1$ and $|Z|=0$. Thus $e_{G-f-Z}(\pi)\leq -1$ by (\ref{equ::7}), which also leads to a contradiction. Therefore, $n_{0}'\geq 2$, as required.

Note that $\rho(G)\geq\rho(B_{n,\delta+1}^{3})>\rho(K_{n-\delta-1})=n-\delta-2$. Combining this with Lemmas \ref{lem::2.5} and \ref{lem::2.6}, we have
\begin{equation*}
m > \frac{n^2}{2}-\frac{(2\delta+3)n}{2}+(\delta+1)^{2}.
\end{equation*}
Since $G$ is a 3-connected graph, it follows that
\begin{equation}\label{equ::12}
|\partial_{G}(V_i)|\geq 3-|Z|.
\end{equation}
Recall that $0\leq |Z|\leq 2$. Then we divide the proof into the following cases.

{\flushleft\bf Case 1.} $0\leq |Z|\leq 1$.

If $n_{0}'=2$, then the partition $\pi$ consists of two nontrivial parts $V_1,V_2$ and $n_0$ trivial parts. Putting (\ref{equ::12}) into (\ref{equ::9}), we get
$$e_{G-Z}(\pi)\geq \frac{1}{2}(|\partial_{G}(V_1)|+|\partial_{G}(V_2)|+6n_0-n_{Z}(\pi))\geq 3-|Z|+3n_0-\frac{1}{2}n_{Z}(\pi)).$$
Combining this with (\ref{equ::8}) and $n_{0}'=2$, we have
$$-n_{0}-\frac{1}{2}n_{Z}(\pi)-|Z|\geq 0,$$
and so $n_{0}=0$, $n_{Z}(\pi)=0$ and $|Z|=0$ by the fact $n_{0}\geq0$, $n_{Z}(\pi)\geq0$ and $|Z|\geq 0$. This implies that the partition $\pi$ consists of two nontrivial parts $V_1,V_2$ and $G-Z=G$. Then $V(G)=V_1\cup V_2$ and $e_{G}(\pi)=e_{G}(V_1,V_2)\leq 3$ by (\ref{equ::8}). Notice that $e_{G}(V_1,V_2)=\frac{1}{2}(|\partial_{G}(V_1)|+|\partial_{G}(V_2)|)\geq 3$ by (\ref{equ::12}). Thus, $e_{G}(V_1,V_2)= 3$. Let $E_{G}(V_1,V_2)=\{f_1,f_2,f\}$. Then we assert that $f_1,f_2,f$ are three independent edges. Otherwise, $G$ is not a 3-connected graph, a contradiction. This implies that $e_{G}(V_1,V_2)=3$ and $G$ is a spanning subgraph of $B_{n,|V_1|}^{3}$, and so
$$\rho(G)\leq\rho(B_{n,|V_1|}^{3}),$$
with equality holding if and only if $G\cong B_{n,|V_1|}^{3}$. Since $\delta\geq 6$ and $e_{G}(V_1,V_2)=3<\delta-1$, it follows that $\min\{|V_1|,|V_2|\}\geq\delta+1$ by Lemma \ref{lem::2.2}. Combining this with Lemma \ref{lem::2.4}, we have
$$\rho(G)\leq\rho(B_{n,\delta+1}^{3}),$$
with equality holding if and only if $G\cong B_{n,\delta+1}^{3}$,
which is impossible because $\rho(G)\geq\rho(B_{n,\delta+1}^{3})$ and $G\ncong B_{n,\delta+1}^{3}$. If $n_0'\geq 3$, by a similar analysis as Case 1 of Theorem \ref{thm::1.1}, we also deduce a contradiction.

{\flushleft\bf Case 2.} $|Z|=2$.

The proof is similar as Case 2 of Theorem \ref{thm::1.1}, so we omit it.\end{proof}

\section{Proof of Theorem \ref{thm::1.3}}

In this section, we give a short proof of Theorem \ref{thm::1.3}. The following lemma is used in the sequel.

\begin{lem}(See \cite{Leman})\label{lem::4.1}
A graph $G$ is a minimally rigid on $n$ vertices and $m$ edges if and only if $m=2n-3$ and
$$e_{G}(X)\leq 2|X|-3$$
for $X\subseteq V(G)$ with $|X|\geq 2$.
\end{lem}

\renewcommand\proofname{\bf Proof of Theorem \ref{thm::1.3}}
\begin{proof}
Suppose that $G$ attains the maximum spectral radius among all minimally rigid graphs for $n\geq 3$, by Lemma \ref{lem::4.1}, we have $m=2n-3$ and $e_{G}(X)\leq 2|X|-3$
for all $X\subseteq V(G)$ with $|X|\geq 2$. Notice that $K_2\nabla (n-2)K_{1}$ is a minimally rigid graph. Then
\begin{equation}\label{equ::13}
\begin{aligned}
\rho(G)\geq \rho(K_2\nabla (n-2)K_{1})=\frac{1+\sqrt{8n-15}}{2}.
\end{aligned}
\end{equation}
We assert that $\delta(G)\geq 2$. Otherwise, there exists a vertex $u$ such that $d_{G}(u)=1$. Note that $m=2n-3$. Then $e_{G}(G-u)=2|V(G-u)|-2$, which is impossible because $e_{G}(G-u)\leq 2|V(G-u)|-3$, and so $\delta(G)\geq 2$.
By Lemmas \ref{lem::2.4}, \ref{lem::2.5}, and the fact $\delta\geq 2$, we obtain

\begin{equation}\label{equ::14}
\rho(G)\leq \frac{1}{2}+\sqrt{2m-2n+\frac{9}{4}}=\frac{1+\sqrt{8n-15}}{2}.
\end{equation}
Thus all the equalities hold in (\ref{equ::13}) and (\ref{equ::14}). This implies that $\delta=2$ and $G$ is a graph in which each vertex is of degree either $\delta$ or $n-1$ due to Lemmas \ref{lem::2.4}. If $n=3$, since $\delta(G)=2$, it follows that $G\cong K_2\nabla K_1$, as required. Let $t$ be the number of vertices $u\in G$ such that $d_{G}(u)=n-1$. If $n\geq 4$ and $G\ncong K_2\nabla (n-2)K_1$, then $t=1$ or $t\geq 3$. Note that $\delta=2$. Thus, $2m=3n-3$ if $t=1$ and $2m\geq 5n-9$ if $t\geq 3$, which is impossible because $m=2n-3$ and $n\geq4$. Therefore, $G\cong K_2\nabla (n-2)K_1$.

This completes the proof.\end{proof}

\end{document}